\documentclass[12pt,a4paper]{article}
\usepackage[utf8]{inputenc}
\usepackage[T1]{fontenc}
\usepackage[english]{babel}
\usepackage{amsmath}
\usepackage{amssymb,amstext}
\usepackage{mathrsfs}
\usepackage{amsthm}

\usepackage{multicol}
\usepackage{enumerate}
\usepackage{fancyhdr}

\usepackage{color}
\usepackage{graphicx}
\usepackage{caption}

\usepackage{geometry}

\newtheorem{theorem}{Theorem}

\newtheorem{prop}{Proposition}

\def\bit{\begin{itemize}}
\def\eit{\end{itemize}}
\def\bdes{\begin{description}}
\def\edes{\end{description}}

\def\beq{\begin{equation}}
\def\eeq{\end{equation}}

\def\ben{\begin{enumerate}}
\def\een{\end{enumerate}}

\def\beqar{\begin{eqnarray}}
\def\eeqar{\end{eqnarray}}

\def\beqarr{\begin{eqnarray*}}
\def\eeqarr{\end{eqnarray*}}
\def\ZZ{{\mathbb Z}}       
\def\RR{{\mathbb R}}  

\begin{document}
	\title{\textbf{Central Limit Theorem for a Self-Repelling Diffusion}}
	\author{ Carl-Erik Gauthier\footnote{Contact email: carlgaut@uw.edu or carlerik.gauthier@gmail.com \newline On leave from Universit\'e de Neuch\^atel, Switzerland}\\ Department of Mathematics\\University of Washington, USA}
	\maketitle
	\date{}
	\begin{abstract}
	 We prove a Central Limit Theorem for the finite dimensional distributions of the displacement for the 1D self-repelling diffusion which solves
		\begin{equation*}
		dX_t =dB_t -\big(G'(X_t)+ \int_0^t F'(X_t-X_s)ds\big)dt, 
		\end{equation*} 
		where $B$ is a real valued standard Brownian motion and $F(x)=\sum_{k=1}^n a_k \cos(kx)$ with $n<\infty$ and $a_1,\cdots ,a_n >0$. 
		
		In dimension $d\geq 3$, such a result has already been established by Horv\'ath, T\'oth and Vet\"o in \cite{HTV} in 2012 but not for $d=1,2$. Under an integrability condition, Tarr\`es, T\'oth and Valk\'o conjectured in \cite{TTV} that a Central Limit Theorem result should also hold in dimension $d=1$. 
	\end{abstract}
	
	\textbf{Keywords:} Central Limit Theorem, Self-Repelling Diffusion\\
	\textbf{MSC2010:} 60F05, 60K35, 60H10, 60J60
	\section{Introduction}
	 In this short note, we aim to prove the Central Limit Theorem (denoted by CLT in the sequel) for the finite dimensional distribution of the displacement for the 1D self-repelling diffusion solving
	 \begin{equation}
	 dX_t =dB_t -\big(G'(X_t)+ \int_0^t F'(X_t-X_s)ds\big)dt,\quad X_0=0, 
	 \end{equation} 
	 where $B$ is a real valued standard Brownian motion, $G(x)=\sum_{k=1}^n \big(u_k\cos(kx) +v_k\sin(kx)\big)$ and $F(x)=\sum_{k=1}^n a_k \cos(kx)$ with $n<\infty$ and $a_1,\cdots ,a_n >0$. 
	 
	 Roughly speaking, \textit{self-repelling diffusions} (as considered here) are time continuous stochastic processes which solve an inhomogeneous stochastic differential equation whose drift part is evolving in time according to the whole past history of the process in such a way that it tends to push the diffusing particle away from the most visited sites.  
	 
	 Under the assumptions made on $F$ and $G$, the Law of Large Number has already been established in \cite[Theorem 2 and Remark 1]{CEGMB}, namely 
	 \begin{equation*}
	 \lim_{t\rightarrow \infty} \frac{X_t}{t} = 0 \; a.s.
	 \end{equation*}
	 A question that one may then ask is whether or not a CLT result holds. The purpose of this note is to provide a positive answer to it. 
	 
	 To the author's knowledge it is the first time that such a result is obtained for a 1D self-repelling diffusion, despite being conjectured in 2012 by Tarr\`es, T\'oth and Valk\'o (\cite[Theorem 2 and its remark]{TTV}) under a positivity\footnote{The Fourier transform of $F$ is nonnegative} and an integrability\footnote{$\rho^2:=\int_{-\infty}^{\infty}p^{-2}\hat{F}(p)dp <\infty$} condition. Nevertheless, Horv\'ath, T\'oth and Vet\"o were able to prove in 2012 a CLT in dimensions $d\geq 3$ (\cite[Theorem 2]{HTV}) by using the Kipnis-Varadhan's CLT result for additive functionals via a \textit{graded sector condition} which turns out to fail in dimension $1$. 
	 
	 Before turning to the presentation of the result, let us briefly make a link between the positivity condition from \cite{TTV} and the positivity of the coefficients $a_1,\cdots, a_n$. Let $b\in L^1(\mathbb{R})\cap C^{\infty}(\RR)$ be a function satisfying the positivity condition from \cite{TTV} and let 
	 $$\varphi_{2\pi}(b)(x)=\sum_{n=-\infty}^\infty b(x+2\pi n) $$
	 be the $2\pi-$periodization transform of $b$. It is an exercise in Fourier analysis to show that $$\varphi_{2\pi}(b)(x)=\frac{1}{2\pi}\sum_{k\in\mathbb{Z}}\hat{b}\big(\frac{k}{2\pi}\big)e^{ikx} .$$ 
	  Because $b$ is even, we have
	 $$\varphi_{2\pi}(b)(x)=\frac{\hat{b}(0)}{2\pi}+\frac{1}{\pi}\sum_{k\geqslant 1}\hat{b}\big(\frac{k}{2\pi}\big)\cos(kx) .$$
	 
	 The paper is organised as follows. In Section \ref{s2}, we state the CLT result and present the tools and concepts involved while its proof is presented in Section \ref{s3}.
	 \section{Tools, notation and results}\label{s2}
	 In this section, we introduce the mathematical background that is necessary to present the main result. 
	 
	 Following the same idea as in \cite{CEGMB}, set $U_j(t)=u_j +\int_0^t \cos(jX_s)ds$ and $V_j(t)=v_j+ \int_0^t \sin(jX_s)ds$. With these new variables, we obtain the following system
	 \begin{equation}
	 \left\{ 
	 \begin{array}{l}
	 dX_{t}=dB_t+\sum_{j=1}^n ja_j \Big(\sin(jX_{t})U_{j}(t)-\cos(jX_{t})V_{j}(t)\Big) dt\medskip \\ 
	 dU_{j}(t)=\cos(jX_{t})dt,\; j=1,\ldots ,n.\medskip\\
	 dV_{j}(t)=\sin(jX_{t})dt,\; j=1,\ldots ,n.
	 \end{array}
	 \right.   \label{EqDECoupl1}
	 \end{equation}
	 Since for all $j=1\cdots ,n$, the functions $x\mapsto \cos(jx)$ and $x\mapsto \sin(jx)$ are $2\pi$-periodic, we can replace $X_t$ by $\Theta_t =X_t$ (mod $2\pi$ ) $\in S^1$, where $S^1= \RR/2\pi\mathbb{Z}$ denotes the 1-dimensional flat torus. This replacement allows us to use the framework from \cite{CEGMB}. 
	 
	 In order to shorten the notation, we let $U_t$ and $V_t$ denote the vectors
	 \begin{equation*}
	 U_t=\big(U_1(t),\cdots , U_n(t)\big)\qquad \text{ and }\qquad V_t=\big(V_1(t),\cdots , V_n(t)\big).
	 \end{equation*}
	 Summarizing the main results from \cite{CEGMB}, we have
	 \begin{theorem} [Theorem 5 and Theorem 6, \cite{CEGMB}])\label{rappel} Let $(P_{t})_{t\geqslant 0}$ be the semi-group associated to the process $\Big((\Theta_t, U_t,V_t)\Big)_{t\geq 0}$ and $P_{t}((\theta_0 ,u_0,v_0),d\theta dudv)$ its transition probability. Then
	 	 \begin{enumerate}
	 	 		\item[1)] The unique invariant probability measure is $$\mu (d\theta dudv)=\nu(d\theta )\otimes \frac{e^{-\Phi(u,v)}}{C}dudv,$$
	 	 		where $\Phi(u,v)=\frac{1}{2}\sum_{k=1}^{n}a_{k}k^2 (u_{k}^{2}+v_k^2)$, $C$ is a normalization constant and $\nu(d\theta)$ 
	 	 		is the uniform probability measure on $S^1=\RR/2\pi\ZZ$.
	 	 		\item[2)] Let $\mu_t =\mathcal{L}(\Theta_t, U_t,V_t)$ denote the law of $(\Theta_t, U_t,V_t) $. Then for any initial distribution $\mu_0$, $\mu_t$ converges to $\mu$ in total variation. 
	 	 		\item[3)] For every $\eta >0$ and $g\in L^{2}(\mu)$
	 	 		$$\Vert P_{t}g-\int g(\theta ,u,v)\mu (d\theta dudv)\Vert_{L^{2}(\mu)}\leqslant \sqrt{1+ 2\eta}\Vert g-\int g(\theta ,u,v)\mu (d\theta dudv)\Vert_{L^{2}(\mu)}e^{-\lambda t}, $$
	 	 		where $$\lambda=\frac{\eta}{1+\eta}\frac{K_{1}}{1+K_{2}+K_{3}}\:,$$
	 	 		with explicit constants $K_1,K_2$ and $K_3$.
	 	 \end{enumerate}
	 \end{theorem} 
	 In this paper, we will adopt the same point of view as in \cite{TTV}: \textit{the environment seen from the particle}. For that purpose, introduce the following new variables
	 \begin{equation} \label{defCj}C_j (t)=U_j(t)\cos(jX_t)+V_j(t)\sin(jX_t)=\Big\langle \begin{pmatrix}
	 U_{j}(t)\\
	 V_{j}(t)
	 \end{pmatrix},\begin{pmatrix}
	 \cos(jX_t)\\
	 \sin(jX_t)
	 \end{pmatrix}\Big\rangle\end{equation} 
	 and 
	 \begin{equation}
	 \label{defSj}
	 S_j(t)=\sin(jX_t)U_j(t)-\cos(jX_t)V_j(t)=\Big\langle \begin{pmatrix}
	 U_{j}(t)\\
	 V_{j}(t)
	 \end{pmatrix},\begin{pmatrix}
	 \sin(jX_t)\\
	 -\cos(jX_t)
	 \end{pmatrix}\Big\rangle.\end{equation} 
	 So, if we denote by $\eta_t$ the potential viewed from the particle's position, i.e 
	 \begin{equation*}
	 \eta_t (x)=F_t(x+X_t)=\int_0^t \sum_{k=1}^n a_k \cos(k(x+X_t -X_s))ds + G(x+X_t),
	 \end{equation*}
	 then 
	 \begin{equation*}
	 \eta_t(x)= \sum_{k=1}^n a_k \Big( C_k(t) \cos(kx)-S_k(t)\sin(kx) \Big).
	 \end{equation*}
	 Moreover, this allows us to rewrite $X_t$ as 
	 \begin{equation*}
	 X_t =B_t +\int_0^t \sum_{k=1}^n ka_kS_k(u)du= B_t+\int_0^t \varphi(\eta_u')du,
	 \end{equation*}
	 where $\varphi:\Omega \rightarrow \RR$ is defined by $\varphi(\omega)=\omega(0)$ and $\Omega$ is the vector space spanned by the functions $\cos(kx)$ and $\sin(kx)$ for $k=0,1,\cdots n$.
	 
	 Before turning to the results, let us introduce the following notation.
 We denote by $T_t$ the semigroup induced by the process $$((C_t,S_t))_{t\geqslant 0}:=\Big(\big(C_1(t),S_1(t),\cdots, C_n(t),S_n(t)\big)\Big)_{t\geqslant 0}$$ and by $G$ its infinitesimal generator. For an operator $R$, we denote its domain by $D(R)$.\vspace{2mm}
 
 Given a probability measure $\pi$ over $\RR^{2n}$, we denote by $L^2(\pi)$ the space $L^2(\RR^{2n},\pi)$, by $\Big\langle .,.\Big\rangle_{L^2(\pi)} $ the associated inner product and by $\Big\Vert . \Big\Vert_{L^2(\pi)}$ the induced $L^2-$norm.
Finally, we denote by $\mathbb{P}_\pi$ the law of the process with initial distribution $\pi = \mathcal{L}\big((C_0,S_0)\big)$ and by $\Longrightarrow$ the convergence in distribution.
	 \begin{prop}\label{resultatinduit}\
	 	
	 	\begin{enumerate}
	 	\item For any smooth function $f$ having compact support, we have 
	 	\begin{eqnarray*}
	 	Gf(c,s)&=&\frac{1}{2}\sum_{j=1}^nj^2\big(s_j^2\partial_{c_jc_j}f+c_j^2\partial_{s_js_j}f\big)-\frac{1}{2}\sum_{k\neq j}jk\big(s_js_k\partial_{c_kc_j}f + c_jc_k\partial_{s_js_k}f\big) \notag\\ & &- \sum_{j,k=1}^njks_jc_k\partial_{c_js_k}f
		  + \Big(\sum_{k=1}^n ka_ks_k\Big)\sum_{j=1}^n j\big(-s_j\partial_{c_j}f+c_j\partial_{s_j}f\big) +\sum_{j=1}^n \partial_{c_j}f.
	 	\end{eqnarray*}
	 	\item The process $((C_t,S_t))_{t\geqslant 0}$ admits a unique invariant probability measure of the form \begin{equation*}
	 	\pi(dcds)= \frac{e^{-\Phi(c,s)}}{C}dcds
	 	\end{equation*}
	 	where $\Phi(c,s)=\frac{1}{2}\sum_{k=1}^{n}a_{k} k^2 (c_{k}^{2}+s_k^2)$ and $C$ is the normalizing constant.
	 	\item For any function $f\in L^2(\pi)$, we have 
	 	\begin{equation*}
	 	\Big\Vert T_t f - \int_{\RR^{2n}}f(c,s)\pi (dcds)\Big\Vert_{L^2(\pi)}\leq \sqrt{3}\Big\Vert f-\int_{\RR^{2n}}f(c,s)\pi (dcds)\Big\Vert_{L^{2}(\pi)}e^{-\lambda t}, 
	 	\end{equation*}
	 	where $\lambda=\frac{1}{2}\frac{K_{1}}{1+K_{2}+K_{3}}$
	 	and the constants $K_1,K_2$ and $K_3$ are those from Theorem \ref{rappel}.
	 \end{enumerate}
	 \end{prop}
	 \begin{proof}\
	 	
	 	\begin{enumerate}
	 		\item It follows from It\^o's formula and \cite[Propositions VII.1.6 and VII.1.7]{RY}.
	 		\item The fact that $\pi(dcds)$ is an invariant probability measure follows from Theorem \ref{rappel} as well as from the equations \eqref{defCj} and \eqref{defSj}. Indeed, for any $A_j\in\mathcal{B}(\RR^2)$, we have by rotation invariance of the Gaussian measure
	 		\begin{eqnarray*}
	 		\pi\Big(A_1\times \cdots \times A_n\Big)&=& \mathbb{P}\Big(\big(C_j(t),S_j(t)\big)\in A_j\, \forall j=1,\cdots, n\Big)\\
	 		&=& \mathbb{P}\Big(\big(U_j(t),V_j(t)\big)\in A_j\, \forall j=1,\cdots, n\Big) \\
	 		&=& \mathbb{P}\Big(\Theta_t \in \mathbb{S}^1, \; \big(U_j(t),V_j(t)\big)\in A_j\, \forall j=1,\cdots, n\Big)\\
	 		&=&\mu\Big(\mathbb{S}^1 \times A_1\times \cdots \times A_n\Big)
	 		\end{eqnarray*}
	 		Concerning the uniqueness, let $\nu$ be an invariant probability measure for the process $((C_t,S_t))_{t\geqslant 0}$. Then define on $S^1\times \RR^n\times\RR^n$ the probability measure $\mu_0(d\theta dudv)= \delta_{0}\otimes \nu(dudv)$ and sample $(\Theta_0, U_0,V_0)$ according to $\mu_0$.
	 		
	 		By Theorem \ref{rappel}, $\mu_t$ converges to $\mu$ in total variation. In particular the marginal law of $\mu_t$ corresponding to $(U_t,V_t)$ converges to $\pi$. Thus $\nu =\pi$.
	 		\item  Let $f:\RR^{n}\times \RR^{n}\rightarrow \RR$ and define a function $g:S^1\times \RR^{n}\times \RR^{n}\rightarrow \RR$ by $$g(\theta,u,v)=f(c,s), $$
	 		where the pairs $(c_j,s_j)$ are defined as in \eqref{defCj} and \eqref{defSj}. Applying It\^o's formula to \eqref{defCj} and \eqref{defSj} yields
	 		\begin{equation*}\label{EqDecoupl2}
	 		d\begin{pmatrix}
	 		C_{1}(t)\\
	 		S_{1}(t)\\
	 		C_2(t)\\
	 		S_2(t)\\
	 		\vdots\\
	 		C_n(t)\\
	 		S_n(t)
	 		\end{pmatrix}= \begin{pmatrix}
	 		-S_{1}(t)\\
	 		C_{1}(t)\\
	 		-2S_2(t)\\
	 		2C_2(t)\\
	 		\vdots\\
	 		-nS_n(t)\\
	 		nC_n(t)
	 		\end{pmatrix}\big( dB_t +(\sum_{k=1}^n ka_k S_k(t))dt\big)- \begin{pmatrix}
	 		C_{1}(t)\\
	 		S_{1}(t)\\
	 		4C_2(t)\\
	 		4S_2(t)\\
	 		\vdots\\
	 		n^2C_n(t)\\
	 		n^2S_n(t)
	 		\end{pmatrix}dt+ \begin{pmatrix}
	 		1\\
	 		0\\
	 		1\\
	 		0\\
	 		\vdots\\
	 		1\\
	 		0
	 		\end{pmatrix}dt .
	 		\end{equation*}
	 		Thus the evolution of $(C_t,S_t)$ does not depend on the dynamic of $\Theta_t$. Therefore
	 		\begin{eqnarray*}
	 		T_tf(c,s)&=&\mathbb{E}\Big(f(C_t,S_t)\mid C_0 =c,\, S_0=s \Big)\\
	 		& =& \mathbb{E}\Big(f(C_t,S_t)\mid \Theta_0=\theta,\, C_0 =c,\, S_0=s \Big)\\
	 		& =& \mathbb{E}\Big(f(C_t,S_t)\mid \Theta_0=\theta,\, U_0 =u,\, V_0=v \Big)\\
	 		& =& \mathbb{E}\Big(g(\Theta_t,U_t,V_t)\mid \Theta_0=\theta,\, U_0 =u,\, V_0=v \Big)\\
	 		&=& P_tg(\theta, u,v),
	 		\end{eqnarray*}
	 		where the pairs $(u_k,v_k)$ are such that $c_k=u_k\cos(k\theta)+ v_k\sin(k\theta)$ and $s_k = u_k\sin(k\theta)- v_k\cos(k\theta)$. 
	 		
	 		The statement follows then from Theorem \ref{rappel} with $\eta=1$.
	 	\end{enumerate}
	 \end{proof}
	 Our CLT result is the following
	 \begin{theorem}\label{mainResult}\
	 	
	 	\begin{enumerate}
	 		\item $\sigma^2:=\lim_{t\rightarrow\infty}\frac{\mathbb{E}_\pi(X_t^2)}{t}$ exists and it satisfies \begin{equation}\label{bounds} 1\leq \sigma^2  \leq 1+ 2\Big(\sum_{j=1}^n \frac{a_j}{j^2}\Big).\end{equation}
	 		\item For any $0< t_1<\cdots < t_n<\infty$, we have 
	 		\begin{equation}
	 		\Big(\frac{X_{t_1/\varepsilon}}{\sigma^2 t_1/\varepsilon},\cdots, \frac{X_{t_n/\varepsilon}}{\sigma^2 t_n/\varepsilon}\Big) \overset{\varepsilon\rightarrow 0}{\Longrightarrow} \Big(W_{t_1},\cdots, W_{t_n}\Big)
	 		\end{equation}
	 		under $\mathbb{P}_\pi$, where $W$ is a real valued standard Brownian motion.
	 	\end{enumerate}
	 \end{theorem}
	  \section{Proof of Theorem \ref{mainResult}}\label{s3}
	Throughout this section, we let $g,h:\RR^n\rightarrow \RR$ denote the functions defined by
	 \begin{equation*}
	 g(c,s) = \sum_{k=1}^n ka_k s_k\hspace{5mm} \text{ and }\hspace{5mm}  h(c,s)=\sum_{k=1}^n a_kc_k.
	 \end{equation*} 
	 \subsection{Proof of Part 1}
	First of all, by repeating the arguments of \cite{TTV}, we have as at the beginning of \cite[Section 4]{TTV} 
	 \begin{eqnarray}
	 \mathbb{E}_\pi(X_t^2)&=& t + \mathbb{E}_\pi\Big(\big(\int_0^t \sum_{k=1}^n ka_k S_k(u)du\big)^2\Big)\label{infoInitial}\\
	 &=& t+ 2\int_0^t (t-u)  \mathbb{E}_\pi\Big(\big( \sum_{k=1}^n ka_k S_k(u)du\big)\big(\sum_{k=1}^n ka_k S_k(0)\big)\Big).\notag\\
	 &=& t + 2\int_0^t (t-u) \Big\langle T_ug,g\Big\rangle_{L^2(\pi)} \label{infoInitial2}.
	 \end{eqnarray} 
	 By the Cauchy-Schwarz inequality and the third part of Proposition \ref{resultatinduit}, $\Big\langle T_ug,g\Big\rangle_{L^2(\pi)}$ decreases exponentially fast to $0$. Hence $\int_0^\infty \Big\langle T_ug,g\Big\rangle_{L^2(\pi)}du$ exists and, therefore, it yields 
	 \begin{equation}
	 \lim_{t\rightarrow\infty}\frac{\mathbb{E}_{\pi}(X_t^2)}{t}=1 + \int_0^\infty \Big\langle T_ug,g\Big\rangle_{L^2(\pi)} du:= \sigma^2.
	 \end{equation}  
	 Now that the existence of $\sigma^2$ is established, let us prove the bounds in \eqref{bounds}. The lower bound is trivial since it follows from \eqref{infoInitial}. In order to establish the upper bound, we follow the arguments presented in \cite{OLLA} based on the Kipnis-Varadhan's CLT theorem.
	 
	 Letting $G^\ast$ denote the adjoint of $G$ in $L^2(\pi)$, $S=\frac{1}{2}(G+G^\ast)$ and $A=\frac{1}{2}(-G+ G^\ast)$ denote the symmetric and the skew-symmetric part of $G$, we have for any smooth function $f$ having compact support
	 \begin{equation}
	 \big\langle Gf , f\big\rangle_{L^2(\pi)}=\big\langle Sf , f\big\rangle_{L^2(\pi)}= -\frac{1}{2}\int \Big(\sum_{j=1}^n j(s_j\partial_{c_j}f-c_j\partial_{s_j}f)\Big)^2 d\pi.
	 \end{equation} 
	 Hence, using the notation of \cite{OLLA}, we have 
	 \begin{equation}
	 \big\langle Gf , f\big\rangle_{L^2(\pi)}=-\Vert f\Vert_1^2.
	 \end{equation} 
	 Because
	 \begin{eqnarray}
	 \int h(c,s) \big(\sum_{j=1}js_j \partial_{c_j}f - jc_j \partial_{s_j}f \big) d\pi &=&\int \big(\sum_{j=1}^n js_j\partial_{c_j}h-jc_j \partial_{s_j}h\big)f d\pi \notag\\
	 & & -\int hf \big(\sum_{k=1}k^2 a_k c_k\big)d\pi\notag\\
	 & & + \int hf \big(\sum_{k=1}k^2 a_k c_k\big)d\pi\notag\\ 
	 &=& \int gfd\pi
	 \end{eqnarray}
	 it follows from the Cauchy-Schwarz inequality that 
	 \begin{equation}
	 \big\vert\int gfd\pi \big\vert \leq \Vert h\Vert_{L^2(\pi)} \Vert f\Vert_1.
	 \end{equation}
	 Hence, with the notation of \cite{OLLA},
	 \begin{equation}
	 \Vert g\Vert_{-1} \leq \Vert h\Vert_{L^2(\pi)} = \sqrt{\sum_{j=1}^n\frac{a_j}{j^2}}.
	 \end{equation}
	 Thus, the upper bound comes from Eq. (2.1.7) in \cite{OLLA}. 
  \subsection{Proof of Part 2}
  Set $\mathcal{S}_t= \int_0^t g(C_u,S_u)du$ and denote by $\mathfrak{s}_t^2$ its variance under $\pi$. \\
  
  Since $\Vert T_tg\Vert_{L^2(\pi)}$ decreases exponentially fast to $0$ by Proposition \ref{rappel} and $\int gd\pi =0$, then, by \cite[Corollary 3.2.i]{CCG}, there exists a function $f\in D(G^{-1})$ such that $Gf=g$ and 
  \begin{equation*}
  \mathfrak{s}_t^2 = 2t \Vert f\Vert_1 = -\Big\langle g, f\Big\rangle_{L^2(\pi)}. 
  \end{equation*}
  
  Because $\mathbb{E}_\pi\Big(\mathcal{}S_t\Big)=0$ for all $t\geq 0$, then $\Vert f\Vert_1 =0$ implies that the process $t\mapsto \mathcal{S}_t$ is the null process $\mathbb{P}_\pi$-almost surely and the result is therefore trivial in that case.
  
  If $\Vert f\Vert_1>0$, then by \cite[Theorem 3.1]{CCG}, we have that for any $0< t_1<\cdots < t_n<\infty$
  \begin{equation}
  \Big(\frac{S_{t_1/\varepsilon}}{s_{t_1/\varepsilon}},\cdots, \frac{S_{t_n/\varepsilon}}{s_{t_n/\varepsilon}}\Big) \overset{\varepsilon\rightarrow 0}{\Longrightarrow} \Big(W_{t_1},\cdots, W_{t_n}\Big)
  \end{equation}
  under $\mathbb{P}_\pi$, where $W$ is a real valued standard Brownian motion.
  
  Thus, the second part of Theorem \ref{mainResult} follows from the martingale approximation in the Kipnis-Varadhan approach as presented in \cite{CCG}.
  \section*{Acknowledgement}\
  
  I acknowledge financial support from the Swiss National Foundation for Research through the grants $200021\_163072$ and P$2$NEP$2\_171951$. 
  
  I am grateful to B\'alint T\'oth for asking me wether or not a CLT result holds in the periodic case considered in this work.

\end{document}